\documentclass[reqno]{amsart}
\usepackage[utf8]{inputenc}
\usepackage{amsmath, amssymb, amsthm, epsfig}
\usepackage{hyperref, latexsym}
\usepackage{url}
\usepackage[mathscr]{euscript}
\usepackage{mathrsfs}
\usepackage{ dsfont }
\usepackage{color}
\usepackage{fullpage} 
\usepackage{setspace}
\usepackage[left=3cm, right=3cm, bottom=3cm]{geometry}     
\usepackage{verbatim} % this enables the {comment} environment
\usepackage{bbm}

\theoremstyle{plain}
\newtheorem{theorem}{Theorem}
\newtheorem{corollary}[theorem]{Corollary}
\newtheorem{lemma}[theorem]{Lemma}

\theoremstyle{definition}

\numberwithin{equation}{section}
\newtheorem*{theorem*}{Theorem}

\newcommand{\R}{{\mathbb R}}

\newcommand{\C}{{\mathbb C}}

\def\Xint#1{\mathchoice
{\XXint\displaystyle\textstyle{#1}}
{\XXint\textstyle\scriptstyle{#1}}
{\XXint\scriptstyle\scriptscriptstyle{#1}}
{\XXint\scriptscriptstyle\scriptscriptstyle{#1}}
\!\int}
\def\XXint#1#2#3{{\setbox0=\hbox{$#1{#2#3}{\int}$}
\vcenter{\hbox{$#2#3$}}\kern-.5\wd0}}

\def\dashint{\Xint-}

\newcommand{\M}{\mathcal{M}}

\newcommand{\dy}{\, \mathrm{d}y}
\newcommand{\dx}{\, \mathrm{d}x}

\newcommand{\dz}{\, \mathrm{d}z}

\renewcommand{\S}{\mathcal{S}}
\renewcommand{\d}{\mathrm{d}}
\renewcommand{\dx}{\,\text{\rm d}x}
\renewcommand{\dy}{\,\text{\rm d}y}
\renewcommand{\dz}{\,\text{\rm d}z}

\newcommand{\dw}{\,\text{\rm d}w}

\DeclareMathOperator{\supp}{supp}

\providecommand{\norm}[1]{ \lVert#1  \rVert}
\providecommand{\sint}{ \dashint }

\title{Regularity of maximal functions on Hardy--Sobolev spaces}

\author[Pérez]{Carlos Pérez}
\author[Picon]{Tiago Picon}
\author[Saari]{Olli Saari} 
\author[Sousa]{Mateus Sousa }

\date{\today}

\address{Departamento de Matemáticas, Universidad del País Vasco UPV/EHU, IKERBASQUE,
Basque Foundation for Science, and BCAM, Basque Center for Applied Mathematics, Bilbao, Spain.}
\email{carlos.perezmo@ehu.es}

\address{University of São Paulo
Faculdade de Filosofia, Ciências e Letras de Ribeirão Preto
Departamento de Computação e Matemática
Avenida Bandeirantes 3900, 1404-040, Ribeirão Preto, Brazil}
\email{picon@ffclrp.usp.br}

\address{Institute of Mathematics, 
	University of Bonn,
	Endenicher Allee 60, 53115, Bonn,
	Germany}
\email{saari@math.uni-bonn.de}

\address{
IMPA - Instituto de Matem\'{a}tica Pura e Aplicada\\
Rio de Janeiro - RJ, Brazil, 22460-320.}
\email{mateuscs@impa.br}


\allowdisplaybreaks
\begin{document}

\begin{abstract}  We prove that maximal operators of convolution type associated to smooth kernels are bounded in the homogeneous Hardy--Sobolev spaces $\dot{H}^{1,p}(\R^d)$ when $1/p < 1+1/d$. This range of exponents is sharp. As a by-product of the proof, we obtain similar results for the local Hardy--Sobolev spaces $\dot{h}^{1,p}(\R^d)$ in the same range of exponents.
\end{abstract}

\maketitle

\section{Introduction}

Let $\varphi:\R^d\rightarrow \R$ be a nonnegative function such that
\begin{equation*}
    \int_{\R^d}\varphi(x)\dx=1.
\end{equation*}
The maximal operator associated to $\varphi$ is defined as
\begin{equation*}
    \mathcal{M}_{\varphi}f(x):=\sup_{t>0}\varphi_t\ast |f|(x),
\end{equation*}
where $\varphi_t(x)=t^{-d}\varphi(\frac{x}{t})$, and $f\in L^1_{loc}(\R^d)$. The simplest example of such an operator is the Hardy--Littlewood maximal operator, which from this point on we denote by $M$. It occurs when $\varphi=\frac{1}{|B_1|}\mathds{1}_{B(0,1)}$, where $B(x,r)$ denotes the $d$-dimensional ball of radius $r$ centered at the $x\in\R^d$ and $|B_r|$ its Lebesgue measure. The operator $M$ evaluates the supremum of all averages of $|f|$ on balls centered at $x$, and for different functions $\varphi$, the operator $\M_\varphi$ can be interpreted as a weighted average variant of $M$.

%Maximal operator are classical objects in analysis, and they satisfy certain $L^p$ bounds. For $1<p\leq \infty$, there is $C=C_{p,d}>0$ such that  
%begin{equation*}
    %\|Mf\|_{L^p(\R^d)}\leq C_{p,d}\|f\|_{L^p(\R^d)},
%\end{equation*}
%and although $Mf\notin L^1(\R^d)$ for $f\neq 0$, there is a weak bound in this case, i.e, for $f\in L^1(\R^d)$ 
%\begin{equation*}
%    m(\{x
%    \in\R^d~:~Mf(x)>\lambda\})\leq %\frac{C_d}{\lambda}\|f\|_{L^1(\R^d)}.
%\end{equation*}
%For general $\varphi$, whenever it admits a radial non-increasing majorant with integral $A$, $\M_\varphi$ is dominated by $M$ (See \cite[Chapter III]{BigStein}), i.e, 
%\begin{equation*}
%    \M_\varphi f(x)\leq A Mf(x),
%\end{equation*}
%and this last inequality implies the same $L^p$ boundedness for $\M_\varphi$, and in case $\varphi$ is radial non-decreasing one also has $\M_\varphi f\notin L^1(\R^d)$ for $f\neq 0$.

It was established by Kinnunen \cite{Kinnunen97} that, for $p>1$, $M$ defines a bounded operator in the Sobolev spaces $W^{1,p}(\R^d)$, i.e, there is $C=C_{p}>0$ such that
\begin{equation}\label{Sobolev_bound}
    \|M f\|_{W^{1,p}(\R^d)}\leq C\|f\|_{W^{1,p}(\R^d)}.
\end{equation}
In his paper \cite{Kinnunen97}, Kinnunen obtains the bound \eqref{Sobolev_bound} by proving that a function $f\in W^{1,p}(\R^d)$ satisfies, for almost every $x\in\R^d$, that 
\begin{equation}\label{equation:DM_smaller_than_MD}
    |\partial_jMf(x)|\leq M(\partial_j(f))(x),
\end{equation}
and this last inequality readily implies 
\begin{equation}\label{Sobolev_derivative_bound}
    \sum_j\|\partial_jMf\|_{L^p(\R^d)}\leq C_{p,d}\sum_j\|\partial_jf\|_{L^p(\R^d)},
\end{equation}
which, combined with the well known $L^p$-boundedness, implies the $W^{1,p}$-boundedness. Despite the fact Kinnunen's work in \cite{Kinnunen97} is stated in terms of the classical Hardy--Littlewood case, his proof extends to all $\M_\varphi$ of convolution type that are $L^p$-bounded, i.e,
\begin{equation*}
    |\partial_j\M_\varphi f(x)|\leq \M_\varphi(\partial_j(f))(x),
\end{equation*}
and henceforth one has the analogue of \eqref{Sobolev_derivative_bound} for $\M_{\varphi}$, and as a consequence the $W^{1,p}(\R^d)$-boundedness. 
When $p=1$, Kinnunen's result can not hold because of the fact that $\M_\varphi f\notin L^1(\R^d)$, and this completely rules out the possibility of $\M_\varphi f $ belonging to $W^{1,1}(\R^d)$. This means his result is sharp in the sense of range of exponents. Despite that, one could still ask what happens when examining only the derivative level of the inequality, i.e, could $\M_\varphi$ satisfy an inequality like \eqref{Sobolev_derivative_bound} for $0<p\leq 1$?  A natural way to address what happens in this regime is switching from the Lebesgue $L^p(\R^d)$ spaces to the Hardy spaces $H^p(\R^d)$.

For $0<p\leq \infty$, a distribution $f:\R^d\rightarrow\C$ lies in the Hardy space $H^p(\R^d)$ if its nontangential Poisson maximal function lies in $L^p(\R^d)$. Given a kernel $\psi:\R\rightarrow\C$, the nontangential maximal function associated to $\psi$ of a function $f$ is defined as
\begin{equation*}
    \tilde\M_\psi f(x)=\sup_{|x-y|\leq t}|\psi_t\ast f(y)|,
\end{equation*}
and $f\in H^p(\R^d)$ if $\tilde\M_P f\in L^p(\R^d)$, where
\begin{equation*}
    P(x)=\frac{\Gamma(\frac{d+1}{2})}{\pi^{\frac{d+1}{2}}}\frac{1}{(1+|x|^2)^{\frac{d+1}{2}}}
\end{equation*}
is the Poisson kernel, and we set
\begin{equation*}
    \|f\|_{H^p(\R^d)}:=\|\tilde\M_P f\|_{L^p(\R^d)} .
\end{equation*}
For $p>1$, as a consequence of the $L^p$-boundedness of the nontangential maximal functions and the Lebesgue differentiation theorem, one has $H^p(\R^d)=L^p(\R^d)$. When $0<p\leq 1$ the scenario is different and $H^p(\R^d)$ differs from $L^p(\R^d)$, what makes them natural substitutes for the $L^p$ spaces in this range of exponents.

A distribution $f\in\S'(\R^d)$ belongs to the homogeneous Hardy--Sobolev spaces $\dot{H}^{1,p}(\R^d)$ if for $j=1,\dots,d$, it has a weak derivative $\partial_jf$ in the space $H^p(\R^d)$, and in this case we set
\begin{equation*}
    \|f\|_{\dot{H}^{1,p}(\R^d)}:=\sum_j\|\partial_jf\|_{H^p(\R^d)}.
\end{equation*}
These spaces were first studied by Strichartz \cite{Strichartz90} and, when $ 1/p <1+1/d$, every distribution $f \in \dot{H}^{1,p}(\R^d)$ is known to coincide with a locally integrable function. In particular, one can always make sense of $\M_\varphi f$, as well as its distributional derivatives, whenever $\varphi$ is sufficiently regular, which raises the natural question of boundedness of $\M_\varphi$ in these spaces on this range of exponents. We answer this question for $\varphi\in \mathcal{S}(\R^d)$.
\begin{theorem}\label{theorem:thm1}
Let $\varphi\in \mathcal{S}(\R^d)$ be non-negative, $\int \varphi >0$ and $1/p < 1+1/d$. If $f\in\dot{H}^{1,p}(\R^d )$, then $\M_\varphi f\in \dot{H}^{1,p}(\R^d )$ and there is $C=C(\varphi,p,d)>0$ such that
\begin{equation}\label{Hardy_derivative_bound}
    \|\M_\varphi f\|_{\dot{H}^{1,p}(\R^d)}\leq C\|f\|_{\dot{H}^{1,p}(\R^d)}.
\end{equation}
In particular, $\M_\varphi$ is a bounded operator from $\dot{H}^{1,p}(\R^d )$ to $\dot{H}^{1,p}(\R^d )$.
\end{theorem}

Theorem \ref{theorem:thm1} offers a new way to obtain a derivative level boundedness result as \eqref{Sobolev_bound} which avoids \eqref{equation:DM_smaller_than_MD} and introduces Hardy space regularity of maximal functions into the fold for the first time. There are three main steps in the proof of Theorem \ref{theorem:thm1}. Let $\psi \in C_c^{\infty}(B(0,1))$, such that $\| \nabla \psi \|_{L^{\infty}}$, $\| \psi \|_{L^{\infty}} \leq 1$ and $\psi \geq 0$. Given $x,y\in\R^d$ satisfying $|x-y|\leq t$, one has 
\begin{align}\label{eq:psi_convolution_bound}
\begin{split}
   |(\partial_j \M_{\varphi}f)\ast  \psi_t(y)|
 	&=  |\M_{\varphi}f \ast (\partial_j \psi_t)(y)| \\
 	&\leq 2t^{-d-1} \int_{B(x,{2t})} (\M_{\varphi}f(z) - c)^{+} \, \dz
\end{split}
\end{align}
for any $c \leq \inf_{z \in B(x,{2t})} \M_{\varphi}f(z)$. It is a well known result \cite[Theorem 2.1.4]{Grafakos2008} that there is a constant $C=C(\psi)>0$ such that 
\begin{equation}\label{eq:norm_equivalence}
    \frac{1}{C}\|\tilde\M_\psi f\|_{L^p(\R^d)}\leq \|f\|_{H^p(\R^d)}\leq C\|\tilde\M_\psi f\|_{L^p(\R^d)},
\end{equation}
and in order to obtain Theorem \ref{theorem:thm1}, the first step is the choice of an appropriate $c\in\R$ for each $t$ in
\eqref{eq:psi_convolution_bound}. We then split $B_{2t}$ into two sets, a local and a non-local piece. The second step is the analysis of the local piece and has two main ingredients: a characterization of Hardy--Sobolev spaces by Miyachi \cite{Miyachi1990}, which is given in terms of the following maximal operator
\begin{equation}\label{Myiachi_maximal_function}
    N_pf(x)=\sup_{B\ni x}\inf_{c\in \R}|B|^{1/d}\left(\sint_B|f(y)-c|^p\dy\right)^{1/p},
\end{equation}
and a self-improvement lemma from \cite{LP2005}. The third step is the study of the non-local piece, in which we will get a bound in terms of the nontangential maximal function associated to $\varphi$. At this point the aforementioned quasi-norm equivalence \eqref{eq:norm_equivalence} will come into play. {Lastly, Theorem \ref{theorem:thm1} is sharp in term of the range of exponents, and we show it in the last section.}

As pointed out, Hardy spaces are a natural extension of the Lebesgue spaces when $0<p\leq 1$, and although this result is the first of this kind, another very natural question is that of what happens in the $W^{1,1}$ case. Given a maximal operator $\M_\varphi$, it is possible to extend \eqref{Sobolev_derivative_bound} to $p=1$, in the sense that there is a constant $C>0$ such that
\begin{equation}\label{L1_Sobolev_derivative_bound}
     \|\nabla\M_\varphi f\|_{L^1(\R^d)}\leq C\|\nabla f\|_{L^1(\R^d)}
\end{equation}
for every function $f\in W^{1,1}(\R^d)$. There has been a lot of effort in understanding this question in the last few years, as well as the problem of determining the optimal constant in \eqref{L1_Sobolev_derivative_bound}. The first work in this direction is due to Tanaka \cite{Ta}, who studied the case of $\varphi(x)=\mathds{1}_{[0,1]}(x)$, the one-sided Hardy--Littlewood maximal operator, and obtained \eqref{L1_Sobolev_derivative_bound} with $C=1$. Later, Kurka proved the same result for the one-dimensional Hardy--Littlewood maximal operator, with $C=240.004$. Still in the one-dimensional setting, the same results for the Heat and the Poisson kernels were obtained by Carneiro and Svaiter \cite{CS} with $C=1$. Other interesting results related to the regularity of maximal operators are \cite{AP, BCHP, CFS, CH, CMa, CM, HM, HO, KS2003, Lu1, Lu2, Saari2016, Te}.

Recently, Luiro \cite{Luiro2017} proved that inequality \eqref{L1_Sobolev_derivative_bound} is true in any dimension for the uncentered Hardy--Littlewood maximal function, provided one considers only radial functions. Later Luiro and Madrid \cite{LM2017} extended the radial paradigm to the uncentered fractional Hardy--Littlewood maximal function.
As a straightforward consequence of Theorem \ref{theorem:thm1}, we obtain partial progress towards the understanding of the $W^{1,1}$ scenario.
\begin{corollary}\label{corollary:cor1} Let $\varphi \in \mathcal{S}(\R^d)$, and $\int \varphi >0$. If $f\in W^{1,1}(\R^d )$ and $\partial_1f,\dots,\partial_df\in H^1(\R^d )$, then there is a $C=C(\varphi,d)>0$ such that
\begin{equation}\label{H1_bound}
    \|\nabla\M_\varphi f\|_{L^1(\R^d)}\leq C\|f\|_{\dot{H}^{1,1}(\R^d)}.
\end{equation}
\end{corollary}
In the same spirit of \cite{Luiro2017,LM2017}, Corollary \ref{corollary:cor1} implies that $|\nabla \M_\varphi f|\in L^1(\R^d)$ under stronger conditions than just $f\in W^{1,1}(\R^d)$, which sheds new light on the question if one might have \eqref{L1_Sobolev_derivative_bound} for general $f\in W^{1,1}(\R^d)$.

%%%%%%%%%%%%%%%%%%%%%%%%%%%%%%%%%%%%%%%%%%%%%%%%%%%%%%%%%%%%%%%%%%%%%%%%%%%%%%%%%%%%%%%%%%%%%%%%%%%%%%%%%%%%%%%%%%%%%%%%%%%%%%%%%%%%%%%%%%%%%%%%%%%%%%%%%%%%%%%%%%%%%%%%%%%%%%%%%%%%%%%%%%%%%%%%%%%%%%%%%%%%%%%%%%%%%%%%%%%%%%%%%%%%%%%%%%%%%%%%%%%%%%%%%%%%%%%%%%%%%%%%%%%%%%%%%%%%%%%%%%%%%%%%%%%%%%%%%%%%%%%%%%%%%%%%%%%%%%%%%%%%%%%%%%%%%%%%%%%%

\subsection{A word on forthcoming notation}
We denote by $d \geq 1$ the dimension of the underlying space. We represent the characteristic function of $E$ by $\mathds{1}_{E}$, and averages of $f \in L^{1}(E)$ are denoted as 
\[\frac{1}{|E|} \int_{E} f(y) \dy = \sint_{E} f(y) \dy = f_E  \]
whenever $E$ is a measurable set with finite Lebesgue measure, and reserve the letter $B$ for euclidean balls, with $\alpha B$ meaning the ball with same center and $\alpha$ times the radius . If not otherwise stated, all spaces of functions are defined over the whole $\mathbb{R}^{d}$, e.g. $L^{1} = L^{1}(\mathbb{R}^{d})$. We denote by $\|f\|_{r,\infty}$ the $L^{r,\infty}(\R^d)$ norm of $f$, i.e, the weak $L^r(\R^d)$-norm of $f$. The positive part of a function $f$ is denoted by $(f)^{+} := 1_{\{f > 0\}} f $. We denote $a \lesssim b$ if $a \leq Cb$ for some constant $C>0$, and $a\eqsim b$ if $a\lesssim b$ and $b\lesssim a$. A possible subscript, such as $a \lesssim_p b$, indicates particular dependency on some other value (here $p$). In the proofs, $C$ represents a generic constant, and may change even within a line.

Given a locally integrable function $\varphi \geq 0$, we define the following auxiliary maximal functions (for $f \in L^{1}_{loc}$)
\begin{align*}
\begin{split}
    &\tilde\M^{a}_{\varphi} f (x)= \tilde\M_{\varphi_{1/a}} f (x)	:= \sup_{(y,t) \in \{(y,t): |x-y| < a t \}} | \varphi_t * f(y) |, \quad a > 0 \\
&M_p f(x) 
	= \sup_{r > 0} \left(  \sint_{B(x,r)} |f(y)|^{p} \, \dy \right)^{1/p}, \quad p > 0     
\end{split}
\end{align*}
If $a = 1$ or $p = 1$, they will be suppressed from the notation. Note that the definition of $\tilde\M^{a}_{\varphi}f$ makes sense for $f$ a tempered distribution provided that $\varphi$ is a Schwartz function.

\section{Preliminaries}
Given a function $f\in L^1_{loc}$, let $N_p(f)$ be the maximal function defined in \eqref{Myiachi_maximal_function}. This operator was first considered by Calder\'on \cite{Calderon72}, when $p>1$, to characterize functions with weak derivatives in $L^p$ spaces, and later studied by Miyachi \cite{Miyachi1990}, for $0<p\leq 1$, in order to obtain similar characterizations for Hardy spaces. As our first ingredient, we use his characterization in the form of the next result. 
\begin{lemma}[Calder\'on \cite{Calderon72}, Miyachi \cite{Miyachi1990}]
\label{lemma:Miyachi}
Let $1/p<1 + 1/d $ and $f \in L^{1}_{loc}$. Then
\begin{equation}\label{Miyachi_inequality}
\|N_{p}f \|_{L^{p}} \eqsim \| f \|_{\dot{H}^{1,p}}     
\end{equation}
\end{lemma} 
\begin{proof}
This follows from Theorem 3 (ii) and Theorem 4 (ii) in \cite{Miyachi1990}.
\end{proof}

Our second ingredient is a self-improvement from \cite{LP2005}. Let $r \in (0,\infty)$. Let $\mathcal{B}$ be the family of all Euclidean balls in $\mathbb{R}^{d}$. A functional $a : \mathcal{B} \to [0, \infty]$ satisfies the condition $D_r$ if there is a finite constant $c$ such that
\[\sum_{ i } a(B_i)^{r} |B_i| \leq c^{r} a(B) |B| \]
whenever the sum is over a family of pairwise disjoint sub-balls of $B$. This condition was introduced in \cite{FPW} and further exploited in \cite{MP}.

For any locally finite Borel measure $\mu$, the functional 
\[a(B) = |B|^{1/d} \left( \frac{\mu(B)}{|B|} \right)^{1/q} \]
satisfies $D_{dq/(d- q)}$. In particular,  $r=dq/(d-q)>1$ if and only if $1/q < 1 + 1/d$.

Let $0<q<p$. For every $f \in L^{q}_{loc}$ it always holds
\[\inf_{c \in \mathbb{R}}  \left( \sint_{B} |f - c|^{q} \right)^{1/q} \leq C |B|^{1/n} \left( \sint_{B} |N_{p} f|^{q} \,\dy \right)^{1/q}. \]
Applying Corollary 1.4 in \cite{LP2005} together with Lemma \ref{lemma:Miyachi}, we obtain the following lemma:
\begin{lemma}[Lerner--P\'erez \cite{LP2005}]
\label{lemma:lerner_perez}
Let $p,q > 0$ be such that $1/p < 1/q < 1 + 1/d$ and let $f \in L^{1}_{loc}(\R^d)$ have first distributional derivatives in $H^{p}(\R^d)$. Then 
\[ |B|^{-1/r} \| 1_{B} (f - f_{B}) \|_{{r, \infty}} \leq C |B|^{1/d}  \left( \sint_{2B} |N_{p}f|^{q} \,\dy \right)^{1/q}  \]
for all balls $B$ and $r=\frac{dq}{d-q}$.  
  \end{lemma}

%\subsection{Remark} We state our results when $f$ is a function for the sake of lighter notation, but it is known (see \cite{Miyachi1990}) that whenever the right hand side of \eqref{Miyachi_inequality}, then $f\in L^1_{loc}$. 

\section{Proof of Theorem \ref{theorem:thm1}}

\subsection{The case of compact support}

Let $f\in \dot{H}^{1,p}(\R^d)$ and $\varphi$ be smooth and compactly supported. By a simple dilation argument, we can assume that $\supp(\varphi)\subset B(0,1)$. Fix $j \in \{1, \ldots, d\}$ and $\psi \in C_0^{\infty}(B(0,1))$ with $\| \nabla \psi \|_{L^{\infty}}, \| \psi \|_{L^{\infty}} \leq 1$ and $\psi \geq 0$. As pointed out in \eqref{eq:psi_convolution_bound}, if $|x-y|\leq t$ then
\begin{align*}
\begin{split}
   |(\partial_j \M_{\varphi}f)\ast  \psi_t(y)| & =|(\partial_j( \M_{\varphi}f-c))\ast  \psi_t(y)| \\
 	&=  |(\M_{\varphi}f-c) \ast (\partial_j \psi_t)(y)| \\
 	&\leq t^{-d-1} \int_{B(x,{2t})} (\M_{\varphi}f(z) - c)^{+} \dz
\end{split}
\end{align*}
for any $c  \leq \inf_{z \in B(x,{2t})} \M_{\varphi}f(z)$. Since $ \| \partial_j |f| \|_{H^{p}} \lesssim \| \partial_j f \|_{H^{p}}$ as a consequence of \cite[Theorem 1]{KS2008}, it suffices to prove the claim for $|f|$. Hence, without loss of generality, we may assume that $f \geq 0$. We choose $c = \inf_{z \in B(x,2t)} \M_{\varphi}f$. We set
\begin{align*}
E_1 = \{y \in B(x,2t) : \M_{\varphi}f(y) = \sup_{r < t} \varphi_r * f (y)   \}, \\
E_2 = \{y \in B(x,2t) : \M_{\varphi}f(y) = \sup_{r \geq t} \varphi_r * f (y)\},
\end{align*}
and we proceed to analyze each set separately.

First, we note that 
\begin{equation}
\label{eq:errorterm}
f_{B(x,4t)} - c = f_{B(x,4t)} - \inf_{z \in B(x,2t)} \M_{\varphi}f(z) \lesssim \sint_{B(x,4t)} |f(y)-f_{B(x,4t)}| \,\dy  .
\end{equation}  
Second, since $\M_{\varphi} f \lesssim Mf$, we have
\begin{align*}
 \int_{E_1}  (\M_{\varphi}f(y) - f_{B(x,4t)}  )^{+} \, \dy  
	&\leq \int_{E_1}  \M_{\varphi}( |f - f_{B(x,4t)}|\mathds1_{B(x,4t)})(y) \, \dy    \\
	&\lesssim \int_{E_1}  M( |f - f_{B(x,4t)}|\mathds1_{B(x,4t)})(y) \, \dy.
\end{align*}
In the first inequality above we have used that $\supp(\varphi)\subset B(0,1)$ and the definition of $E_1$. Since for any $h\in L^1_{loc}$ and $r>1$ one has
\begin{equation*}
    \|h\|_{L^1(E)}\leq r'|E|^{1-1/r}\|h\|_{r,\infty},
\end{equation*}
for any choice of $q  \in (d/(d+1),p)$, we have for $r = dq/(d-q)$ that
\begin{align*}
    \begin{split}
     \int_{E_1}  (\M_{\varphi}f(y) - (\M_{\varphi}f)_{B(x,4t)}  )^{+} \,\dy    & \lesssim  |E_1|^{1-1/r} \|  M( |f - f_{B(x,4t)} |\mathds1_{B(x,4t)}) \|_{L^{r, \infty}} 
	  \\ 
	&\lesssim |B_{4t}|^{1-1/r} \| (f - f_{B(x,4t)})\mathds1_{B(x,4t)}   \|_{L^{r, \infty}}  ,
    \end{split}
\end{align*}
the last inequality being due to boundedness of $M$ on $L^{r, \infty}(\R^d)$ when $r > 1$. Now we appeal to Lemma \ref{lemma:lerner_perez} to both the display above as well as to the quantity \eqref{eq:errorterm} to obtain
\begin{equation}\label{eq:E1_estimate}
    \int_{E_1}  (\M_{\varphi}f(y) - c  )^{+} \,\dy\lesssim t^{d+1}\left(\sint_{B(x,8t)}|N_pf(y)|^q\dy\right)^{1/q} \leq t^{d+1}M_q(N_pf)(x)
\end{equation}

We move on to estimate the integral over $E_2$. Let $y \in E_2$, $z \in B(x,2t)$ and $r > t$. Let $\vec{e} = \frac{y-z}{|y-z|}$. Then
\begin{align*}
\varphi_r * f(y) - \varphi_r * f(z) &= \int_{ \mathbb{R}^{d}} f(w)[ \varphi_r (y-w) - \varphi_r (z-w)] \, \dw \\
 	& = \int_{\mathbb{R}^{d}}  \int_{0}^{|y-z|} (f(w) \vec{e}~) \cdot \nabla \varphi_r \left(z - w + \tau \vec{e}  \right) \, \d\tau \dw  \\
 	& \leq|y-z|\sum_{j=1}^d   \sup_{|\eta|\leq |y-z|} |f * \partial_j\varphi_r \left(z + \eta \right)| \\
\end{align*}
since $|(z+\eta) - z| = |\eta| \leq |y-z| \leq 4t<4r$, and $|f\ast\partial_j\varphi|=|\partial_jf\ast\varphi|,$ we have 
\begin{equation*}
    \varphi_r * f(y) - \varphi_r * f(z)\leq 4t\sum_{j=1}^d   \sup_{|w-z|\leq 4r} | \partial_jf\ast\varphi_r \left(w \right)|\leq 4t\sum_{j=1}^d   \tilde{\M}^{4}_\varphi(\partial_jf)(z).
\end{equation*}
Consequently, for $y \in E_2$
\begin{align*}
\M_{\varphi}f(y) - \inf_{z \in B(x,2t)} \M_{\varphi}f(z)
	& = \sup_{r > t} \inf_{z \in B(x,2t)} \sup_{\rho > 0} ( \varphi_r * f(y) -  \varphi_\rho * f(z) ) \\
	& \leq \sup_{r > t} \inf_{z \in B(x,2t)}  ( \varphi_r * f(y) -  \varphi_r * f(z) ) \nonumber \\
 &\leq 4 t \inf_{z \in B(x,2t)}  \tilde{\M}^{4}_\varphi(\partial_j f)(z)\\
 &\leq 4t \tilde{\M}^{4}_\varphi(\partial_j f)(x) ,
\end{align*}
and we conclude 
\begin{equation}
\label{eq:E2_estimate}
\int_{E_2} ( \M_{\varphi}f - c )^{+} \,\dy \leq  4t|E_2|\tilde{\M}^{4}_\varphi(\partial_j f)(x)  \lesssim t^{d+1 } \tilde{\M}^{4}_\varphi(\partial_j f)(x).
\end{equation}
Combining \eqref{eq:psi_convolution_bound}, \eqref{eq:E1_estimate} and \eqref{eq:E2_estimate}  we have
\begin{align*}
 \sup_{|x-y|\leq t}|(\partial_k \M_{\varphi}f)\ast  \psi_t(y)| 
 	&\lesssim M_{q} N_{p}f (x) + \sum_{j=1}^{d} \tilde{\M}^{4}_\varphi(\partial_j f)(x),
\end{align*}
and since $p/q>1$, it from follows boundedness of $M_q$ on $L^{p}$, the already mentioned quasi-norm equivalence \eqref{eq:norm_equivalence} in $H^p$ applied to both $\psi$ and $\varphi_{1/8}$, and Lemma \ref{lemma:Miyachi} that
\[ \norm{ (\M_{\varphi}f ) }_{\dot{H}^{1,p}} \lesssim \norm{N_{p}f }_{L^{p}} + \sum_{j=1}^{d} \norm{\tilde{\M}^{4}_\varphi(\partial_j f)}_{L^p} \lesssim\norm{ f}_{\dot{H}^{1,p}}, \]
which is the desired result.
\subsection{The case of a general support} 
Given $\varphi \in \mathcal{S}(\R^d)$, for some constant $C=C(\varphi)$ one has
\begin{equation*}
    \varphi(x)\leq C \sum_{k=0}^{\infty}2^{-k}\tfrac{1}{|B(0,2^k)|}\mathds1_{B(0,2^k)}.
\end{equation*}
We can proceed now as in the case of compact support and divide $B(x,2t)$ into $E_1$ and $E_2$. In $E_2$ the support does not play a role in the proof. In $E_1$ one just has to observe that
\begin{align*}
    \begin{split}
    (\M_{\varphi}f - f_{B(x,4t)}  )^{+}
    &\lesssim \sum_k 2^{-k}M( |f - f_{B(x,4t)}|\mathds1_{B(x,2^k(4t))}) \\
    &\leq  \sum_k 2^{-k}[M( |f - f_{B(x,2^k(4t))}|\mathds1_{B(x,2^k(4t))})+|f_{B(x,4t)}-f_{B(x,2^k(4t))}|] \\
    &=\sum_k 2^{-k}[M( |f - f_{B(x,2^k(4t))}|\mathds1_{B(x,2^k(4t))})+|(f-f_{B(x,2^k(4t))})_{B(x,4t)}|] \\
    &\lesssim \sum_k 2^{-k}M( |f - f_{B(x,2^k(4t))}|\mathds1_{B(x,2^k(4t))}).
    \end{split} 
\end{align*}
Now integrating over $E_1$ and applying  Lemma \ref{lemma:lerner_perez} in each $B(x,2^k(4t))$ as done before will imply the desired result for Schwartz kernels. 

\section{Remarks}
{
\subsection{The hypothesis on the kernels} One might wonder if the hypothesis $\varphi \in \mathcal{S}(\R^d)$ can be weakened, to obtain, for instance, the same result for the Hardy-Littlewood maximal function. In the analysis of the local piece, smoothness is not used at all, so this part of the analysis holds as long as the kernel has enough decay. On the other hand, the analysis of the non-local piece depends of norm equivalence considerations in $H^p(\R^d)$. To keep technicalities to a minimum we state it only in terms of Schwartz functions, but as long as $\M_\varphi$ satisfies \eqref{eq:norm_equivalence} and decays faster than $(1+|x|)^{-d-1}$, the same techniques apply. Unfortunately, this means the Hardy-Littlewood case falls out of the scope of the techniques employed.}

\subsection{Sharpness of the results in term of the range of exponents} One might wonder if its possible to extend the results in Theorem \ref{theorem:thm1} to the case $1/p\leq 1+1/d$. The answer is negative. If one considers any smooth compactly supported function $f$ with vanishing moments up to order 1, then $f\in \dot{H}^{1,\frac{d}{d+1}}(\R^d)$. On the other, for any kernel $\varphi\in C_c^\infty(\R)$, one has for $j=1,\dots,d$ that $(\partial_j\M_\varphi(f))'(x)\eqsim |x|^{-(d+1)}$ when $|x|\rightarrow\infty$, which implies it does not belong to $L^{\frac{d}{d+1}}(\R^d)$, and therefore $H^{\frac{d}{d+1}}(\R^d)$. To see this is 
true, we use the following observation due to Luiro \cite{Lu1}: if for some $t>0$ one has $\M_\varphi(f)(x)=|f|\ast\varphi_t(x)$ then
$$\partial_j\M_\varphi(f)(x)=\partial_j|f|\ast\varphi_t(x)=\tfrac{1}{t}|f|\ast(\partial_j\varphi)_t(x)$$
Now, when $|x|\rightarrow\infty$, any admissible $t$ will be roughly the size of $|x|$, and now by standart considerations this will imply $\partial_j\M_\varphi(f)(x)\eqsim |x|^{-(d+1)}$.

\subsection{Local Hardy spaces} One can consider similar questions on the local Hardy spaces $h^p(\R^d)$
introduced by Goldberg \cite{Goldberg}. They are defined similarly as the spaces $H^p(\R^d)$, but with a truncated nontangential maximal operator, i.e, $f\in h^p(\R^d)$ when $\tilde\M_P^1 f\in L^p(\R^d)$, where 
\begin{equation*}
    \tilde{m}_Pf(x)=\sup_{|x-y|\leq t\leq 1}|P_t\ast f(y)|.
\end{equation*}
A function belongs to $\dot{h}^{1,p}(\R^d)$ if $\partial_1f,\dots,\partial_df\in h^p(\R^d)$ and we set
\begin{equation*}
    \|f\|_{\dot{h}^{1,p}(\R^d)}:=\sum_j\|\partial_jf\|_{h^p(\R^d)}.
\end{equation*}
If one considers the operator
\begin{equation}
    m_\varphi f(x):=\sup_{0< t\leq 1}\varphi_t\ast |f|(x),
\end{equation}
we have the following result
\begin{theorem}\label{theorem:Thm2}
Let $\varphi\in \mathcal{S}(\R^d)$ and $1/p < 1+1/d$. Then $m_\varphi$ is a bounded operator from $\dot{h}^{1,p}(\R^d)$ to $\dot{h}^{1,p}(\R^d)$.
\end{theorem}

The proof of this result follows the same lines as Theorem \ref{theorem:thm1} since one has the analogue of Lemma \ref{lemma:Miyachi} (see \cite{Miyachi1990}) for the $\dot{h}^{1,p}$ spaces, as well as the norm equivalence with any truncated nontangential maximal operator associated to a Schwartz kernel, so we omit the details.

\section*{Acknowledgments}
The authors are thankful to Emanuel Carneiro for helpful comments and discussions. M.S. acknowledges support from FAPERJ-Brazil. O.S. was supported by the Academy of Finland. Part of the research was done while O.S. was visiting University of the Basque country and also while in residence at Mathematical Sciences Research Institute, Berkeley, California, during the Spring 2017 semester (supported by NSF Grant No.~DMS-1440140). He wishes thank those institutes. C.P is supported by the Basque Government through the BERC 2014-2017 program and by Spanish Ministry of Economy and Competitiveness MINECO through BCAM Severo Ochoa excellence accreditation SEV-2013-0323 and through the project MTM2014-53850-P.

\end{document}